\documentclass[10pt,twoside]{article}
\usepackage{amsfonts,amssymb}
\usepackage{graphicx}
\usepackage{eucal}
\usepackage{mathrsfs}
\usepackage{theorem}
\usepackage{pifont}
\usepackage[intlimits]{amsmath}

\newcommand{\Real}{\mathbb{R}}

\newcommand{\Complex}{\mathbb{C}}

\newcommand{\todo}[1]{{\sffamily To do:}}

\newtheorem{theorem}{Theorem}

\newtheorem {corollary}{Corollary}
\newtheorem {lemma}{Lemma}
\newenvironment{proof}{{\flushleft \emph{Proof}:}}{\ding{110}}

\newenvironment{proof1}{{\flushleft \emph{Proof of Theorem \ref{main}}:}}

\newenvironment{proof2}{{\flushleft \emph{Proof of Corollary \ref{cor1}}:}}

\newenvironment{proof3}{{\flushleft \emph{Proof of Corollary \ref{cor2}}:}}

\title{On finiteness of the sum of negative eigenvalues of Schr\"odinger operators}
\author{Michael Demuth\footnotemark[1],\;
Guy Katriel{\footnotemark[1]\;\;\footnotemark[2]\;\;}}

\date{}

\begin{document}

\maketitle{}
\renewcommand{\thefootnote}{\fnsymbol{footnote}}
\footnotetext[1]{Institute of Mathematics, Technical University of
Clausthal, 38678 Clausthal-Zellerfeld, Germany.}
\footnotetext[2]{Partially supported by the Humboldt Foundation
(Germany).}

\begin{abstract}
We prove conditions on potentials $V$ which imply that the sum of
the negative eigenvalues of the Schr\"odinger operator $-\Delta+V$
is finite. We use a method for bounding eigenvalues based on
estimates of the Hilbert-Schmidt norm of semigroup differences and
on complex analysis.
\end{abstract}


\section{Introduction}
A basic theme in the theory of Schr\"odinger operators $H=-\Delta+V$
is to relate the properties of the potential $V$ to properties of
the set of eigenvalues of $H$. In this paper we prove conditions on
the potential which are sufficient in order that the sum of negative
eigenvalues of the Schr\"odinger operator be finite:
\begin{equation}\label{m}
\sum_{\lambda\in \sigma^{-}(H)} |\lambda|<\infty,
\end{equation}
where $\sigma^-(H)=\sigma(H)\cap(-\infty,0)$, the negative part of
the spectrum of $H$.

Our main result is the following
\begin{theorem}\label{main}
Assume $d\geq 4$. Let $V:\Real^d\rightarrow\Real$ be a Kato
potential, and assume that $V_-=\min(V,0)$ satisfies, for some
$c>0$,
\begin{equation}\label{cond0}
\int_{\Real^d} \int_{\Real^d}e^{-c|w-w'|^2}|V_-(w)||V_-(w')| dw dw'
<\infty.
\end{equation}
and also

\noindent (i) If $d=4$ then
\begin{equation}\label{cond2}
\iint_{|w-w'|<1}\log\Big(\frac{1}{|w-w'|}\Big)|V_-(w)||V_-(w')| dw'
dw <\infty.
\end{equation}
\noindent (ii) If $d\geq 5$ then
\begin{equation}\label{cond1}
\iint_{|w-w'|<1}\frac{|V_-(w)||V_-(w')|}{|w-w'|^{d-4}} dw dw'
<\infty.
\end{equation}
Then (\ref{m}) holds.
\end{theorem}

From the above Theorem we derive the following $L^p$-conditions on
$V$ for (\ref{m}) to hold:
\begin{corollary}\label{cor1} Assume $d\geq 4$ and $V$ is Kato, and
$V_-\in L^p$, where $p\in [\frac{2d}{d+4},2]$. Then (\ref{m}) holds.
\end{corollary}

\begin{corollary}\label{cor2} If $d\geq 4$, $V$ is Kato and $V_-\in L^1$, then (\ref{m}) holds.
\end{corollary}

It is interesting to compare these results with those that can be
obtained from the Lieb-Thirring inequalities, which also give some
$L^p$-conditions implying (\ref{m}). The Lieb-Thirring inequalities
\cite{lieb,laptev}
\begin{equation}\label{lt}\sum_{\lambda\in
\sigma^{-}(H)} |\lambda|^{\gamma}\leq
C_{d,\gamma}\int_{\Real^d}|V_-(x)|^{\frac{d}{2}+\gamma}dx,\end{equation}
hold for $\gamma\geq \frac{1}{2}$ when $d=1$, for $\gamma>0$ when
$d=2$, and for $\gamma\geq 0$ when $d\geq 3$. Since finiteness of
the left-hand side of (\ref{lt}) for any $\gamma\leq 1$ implies
(\ref{m}), we get the following sufficient conditions for (\ref{m})
to hold:

\noindent (1) $d=1$, $V_-\in L^p$, where $p\in [1,\frac{3}{2}]$.

\noindent (2) $d=2$, $V_-\in L^p$, where $p\in (1,2]$.

\noindent (3) $d\geq 3$, $V_-\in L^p$, where $p\in
[\frac{d}{2},\frac{d}{2}+1]$.

Comparing with our Corollaries \ref{cor1}, \ref{cor2}, we see that
in the case $d=4$ the Lieb-Thirring inequalities give (\ref{m}) when
$p\in [2,3]$, while Corollary \ref{cor1} gives the range of values
$p\in [1,2]$, so together we have the range $p\in [1,3]$. In the
case $d\geq 5$ the ranges of values of $p$ for which (\ref{m}) holds
given by Corollary \ref{cor1} are {\it{disjoint}} from the range of
values given by the Lieb-Thirring inequalities. We also note that
the result Corollary \ref{cor2} does not follow from the
Lieb-Thirring results.

An immediate question is whether the results of Corollaries
\ref{cor1},\ref{cor2} hold in dimensions $1,2,3$, that is whether
the restriction $d\geq 4$ that we impose is an artifact of our
method of proof or a reflection of the actual situation. In fact we
can construct a counterexample showing that the result of Corollary
\ref{cor1} is {\it{not}} true when $d=1$. Considering a potential of
the form $V(x)=-(1+|x|)^{-\alpha}$, one has $V\in L^2(\Real)$ when
$\alpha>\frac{1}{2}$. A WKB approximation shows that, when
$\alpha\in (0,2)$, the $n$-th eigenvalue satisfies $\lambda_n \sim
n^{-\frac{2\alpha}{2-\alpha}},$ so that when $\alpha <\frac{2}{3} $
the sum of the eigenvalues diverges. Thus, for $\alpha \in
(\frac{1}{2},\frac{2}{3})$ we have that $V\in L^2(\Real)$, yet the
sum of eigenvalues diverges. On the other in the case of Corollary
\ref{cor2}, the Lieb-Thirring results show that it is also valid for
$d=1$. We do not know whether Corollaries \ref{cor1},\ref{cor2} are
valid in dimensions $d=2,3$.

The technique we use for the proof of Theorem \ref{main} is a
considerable refinement of ideas we introduced in \cite{dk}. There
we developed a method, based on the Jensen identity of complex
analysis, to bound the moments (sums of powers) of the negative
eigenvalues of a self-adjoint operator $B$ on a complex Hilbert
space ${\cal{H}}$, assuming that there is a self-adjoint operator
$A$ with $\sigma(A)\subset [0,\infty)$, such that the semigroup
difference $D_t=e^{-tB}-e^{-tA}$ is a trace class or Hilbert-Schmidt
operator. We obtained some general `abstract' results bounding the
moments of eigenvalues. Applied to Schr\"odinger operators, these
results implied that, under appropriate conditions on the potential,
the moment sum on the left-hand side of (\ref{lt}) is finite for
$\gamma>2$. Theorem \ref{main}, which corresponds to the case
$\gamma=1$, is proven using the same method, but with the difference
that by restricting ourselves to Schr\"odinger operators rather than
general selfadjoint operators, we are able to improve the estimates
in such a way that the stronger result is proven.

We note that from the proof of Theorem \ref{main} one can extract
explicit bounds for the sum of negative eigenvalues, in terms of
 of $V_-$ (Kato norms and the quantities given in
(\ref{cond0}),(\ref{cond2}),(\ref{cond1})). However these
expressions are rather cumbersome, so we have decided to concentrate
on the more `qualitative' aspect of the results.

In the following section we recall the method developed in
\cite{dk}. In Section \ref{proof} we apply the method to obtain the
proof of Theorem \ref{main}.

\section{The Jensen formula and eigenvalues}
\label{mi}

In this section we recall the technique developed in \cite{dk}.
Assume that $A,B$ are self-adjoint operators in a complex Hilbert
space ${\cal{H}}$, with $\sigma(A)\subset [0,\infty)$, $B$
semibounded from below, and that the difference of semigroups
$D_t=e^{-tB}-e^{-tA}$ is Hilbert-Schmidt, for some $t>0$. These
assumptions imply, by Weyl's Theorem, that
$\sigma_{ess}(B)=\sigma_{ess}(A)\subset[0,\infty)$ so the negative
part of the spectrum $\sigma^{-}(B)$ consists only of eigenvalues,
which can only accumulate at $0$.

We define the operator-valued function
\begin{equation}\label{df}F(z)=z[I-ze^{-tA}]^{-1}D_t,\end{equation}
on $\Omega=\{ z\in \Complex \;|\; |z|<1 \}$. Note that the
assumption $\sigma(A)\subset [0,\infty)$ implies that the inverse
$[I-ze^{-tA}]^{-1}$ is well-defined. We have the identity
$$[I-ze^{-A}]^{-1}[I-ze^{-B}]=I-F(z),$$
which implies, for $\lambda<0$,
\begin{equation}\label{fir}\lambda\in \sigma^{-}(B)
 \;\Leftrightarrow \; 1\in
\sigma(F(e^{\lambda})).\end{equation}

The assumption that $D_t$ is Hilbert-Schmidt implies that $F(z)$ is
Hilbert-Schmidt, and we can define the holomorphic function $h(z)$
in $|z|<1$ by
\begin{equation}\label{dh} h(z)=Det_2(I-F(z)),\end{equation}
where $Det_2$ denotes the regularized determinant defined for
Hilbert-Schmidt perturbations of the identity (see e.g.
\cite{simonb}).

 From (\ref{fir}) we have, for $\lambda<0$,
\begin{equation}\label{key}\lambda\in \sigma^{-}(B)\;\Leftrightarrow \;
h(e^{\lambda})=0,\end{equation} and moreover the multiplicity of
$\lambda$ as an eigenvalue of $B$ coincides with multiplicity of
$e^\lambda$ as a zero of $h$.

We now recall the Jensen identity from complex analysis (see e.g.
\cite{rudin}, p. 307).
\begin{lemma}
\label{jensen0} Let $\Omega$ be the open unit disk. Let $h:\Omega
\rightarrow \Complex$ be a holomorphic function, and assume
$h(0)=1$. Then, for $0\leq r<1$,
$$\frac{1}{2\pi}\int_0^{2\pi}\log(|h(r e^{i\theta})|)d\theta=\log\Big(\prod_{|z|\leq r,\; h(z)=0}
\frac{r}{|z|}\Big).$$
\end{lemma}
A variation on a particular case of the results of \cite{dk} which
we will use here is:
\begin{theorem}\label{identity}
Let $A$,$B$ be self-adjoint in a complex Hilbert space ${\cal{H}}$,
with $\sigma(A)\subset[0,\infty)$. Assume that, for some $t>0$,
$D_t=e^{-tB}-e^{-tA}$ is Hilbert-Schmidt. Then, defining $h$ by
(\ref{dh}), we have
\begin{equation}\label{identity1}
\sum_{\lambda\in \sigma^{-}(B)} |\lambda|=
\frac{1}{t}\lim_{r\rightarrow
1-}\frac{1}{2\pi}\int_0^{2\pi}\log(|h(re^{i\theta})|)d\theta.
\end{equation}
\end{theorem}

\begin{proof} By Jensen's identity and (\ref{key}) we have
\begin{eqnarray*}
&&\lim_{r\rightarrow 1-}\frac{1}{2\pi}\int_0^{2\pi}\log(|h(r
e^{i\theta})|)d\theta=\log\Big(\prod_{|z|<1, h(z)=0}
\frac{1}{|z|}\Big)\\
&=&\log\Big(\prod_{\lambda\in \sigma^{-}(B)}
\frac{1}{e^{\lambda}}\Big)=\sum_{\lambda\in \sigma^{-}(B)}
|\lambda|.
\end{eqnarray*}
\end{proof}

Theorem \ref{identity} shows that one can bound the sum of the
negative eigenvalues by bounding the function $h$, and this is our
task now.

Let us first note that, by the general inequality
\begin{equation*}\label{it}|Det_2(I-T)|\leq
e^{\frac{1}{2}\|T\|_{HS}^2}\end{equation*} for Hilbert-Schmidt
operators $T$, we have
\begin{equation}\label{itt}\log(|h(z)|)\leq
\frac{1}{2}\|F(z)\|_{HS}^2,\end{equation} so that we can bound
$h(z)$ by bounding the Hilbert-Schmidt norm of $F(z)$. To do this,
one can - and this is what was done in \cite{dk} - use the
inequality
\begin{equation}\label{ob}\|F(z)\|_{HS}\leq
|z|\|[I-ze^{-tA}]^{-1}\| \|D_t\|_{HS},\end{equation} where the norm
$\|[I-ze^{-tA}]^{-1}\|$ is the regular operator norm, which can in
turn be bounded in terms of the inverse distance of the spectrum of
$I-ze^{-tA}$ to $0$, using the assumption that $\sigma(A)\subset
[0,\infty)$. In this way we obtain the general results of \cite{dk}.

The observation at the basis of this work is that, when the
operators $A$,$B$ are Schr\"odinger operators, the bound on
$\|F(z)\|_{HS}$ obtained by using (\ref{ob}) is not optimal, and one
can obtain better bounds in the Schr\"odinger case by {\it{not}}
separating the estimation into two parts as in (\ref{ob}). For
example the bounds we obtain show that when $d\geq 5$, the function
$h(z)$ is {\it{uniformly}} bounded in the unit disk $|z|<1$, whereas
the bound obtained by using (\ref{ob}) goes to $+\infty$ as
$z\rightarrow 1$. These improved bounds lead, through Theorem
\ref{identity}, to improved bounds on the sum of the negative
eigenvalues of $B$.

\section{Proofs}
\label{proof}

Recall that the potential $V:\Real^d\rightarrow \Real$ is said to
belong to the class $K(\Real^d)$ if
$$\lim_{t\rightarrow 0}\sup_{x\in\Real^d}\int_0^{t}(e^{\eta \Delta
}|V|)(x)d\eta=0.$$ We note that when $d\geq 3$, a necessary and
sufficient condition for $V\in K(\Real^d)$ is that
\begin{equation}\label{kato}\lim_{\alpha\rightarrow 0}\Big[\sup_{x\in\Real^d}
\int_{|y-x|\leq
\alpha}\frac{|V(y)|}{|y-x|^{d-2}}dy\Big]=0.\end{equation} We recall
also that when $d\geq 3$, a {\it{sufficient}} condition for $V\in
K(\Real^d)$ (see \cite{simon}) is that $V$ is uniformly-locally in
$L^p$ for some $p>\frac{d}{2}$ , that is
$$\sup_{x\in \Real^d}\int_{|y-x|\leq 1}|V(x)|^{p}dx<\infty.$$
$V$ is said to belong to class $K^{loc}(\Real^d)$ if $\chi_{Q}V\in
K(\Real^d)$ for any ball $Q \subset \Real^d$, where $\chi_{Q}$
denotes the characteristic function of $Q$. $V$ is said to be a Kato
potential if $V_-=\min(V,0)\in K(\Real^d)$ and $V_{+}=\max(V,0)\in
K^{loc}(\Real^d)$.

By the min-max principle, the eigenvalues of $-\Delta+V_-$ are
smaller than or equal to the corresponding eigenvalues of
$-\Delta+V$, and therefore we have
\begin{equation}\label{pni}\sum_{\lambda \in \sigma^{-}(-\Delta+V)}|\lambda|\leq
\sum_{\lambda \in \sigma^{-}(-\Delta+V_-)}|\lambda|,\end{equation}
so that to prove Theorem \ref{main} it suffices to show that the
right-hand side of (\ref{pni}) is finite. We shall therefore take
$A=H_0=-\Delta$, $B=H_0+V_-$, so that
$$D_t=e^{-t(H_0+V_-)}-e^{-tH_0}.$$

We recall some fundamental facts about Schr\"odinger semigroups (see
e.g. \cite{demuth,simon}), which will be needed below:

\begin{lemma}\label{sch} If $V_-\in K(\Real^d)$ then the Schr\"odinger
semigroup $e^{-t(H_0+V_-)}:L^2(\Real^d)\rightarrow L^2(\Real^d)$
($t\geq 0$) is well defined, and moreover we have, for all $t>0$,
$$\| e^{-t(H_0+V_-)}\|_{L^1,L^\infty}<\infty,$$
$$\sup_{s\in [0,t]}\| e^{-s(H_0+V_-)}\|_{L^\infty,L^\infty}<\infty.$$
\end{lemma}

As explained in the previous section, our task is to bound the norm
$\|F(z)\|_{HS}$, where $F(z)$ is given by (\ref{df}).

We define the operator-valued function $G(z)$, $|z|<1$, by
$$G(z)=ze^{-tA}[I-ze^{-tA}]^{-1}.$$
It is easily checked that
$$[I-ze^{-tA}]^{-1}=I+G(z),$$
hence
$$F(z)=z[I+G(z)]D_t,$$
so that
\begin{eqnarray}\label{fo}
\|F(z)\|_{HS}\leq |z|[\|D_t\|_{HS}+\|G(z)D_t\|_{HS}]
\end{eqnarray}
We are going to bound the two terms on the right-hand side of
(\ref{fo}).

We divide the required estimates into several steps.

\subsection{Some estimates on $G(z)$}
\begin{lemma}
The operator $G(z)$ can be represented in the form
\begin{equation}\label{rep}G(z)f=g_z*f,\;\;\;\forall f\in L^2(\Real^d),\end{equation} where $g_z\in
L^\infty(\Real^d)\cap L^2(\Real^d)$.
\end{lemma}

\begin{proof}
From the definition of $G(z)$ and the properties of the Fourier
transform we have, for $f\in L^2(\Real^d)$,
$$\mathfrak{F}(G(z)f)=ze^{-t
|\xi|^2}[1-ze^{-t|\xi|^2}]^{-1}\mathfrak{F}(f),$$ so that if define
$g_z:\Real^d\rightarrow \Complex$, for $|z|<1$, by
$$g_z=z\mathfrak{F}^{-1}(e^{-t |\xi|^2}[1-ze^{-t|\xi|^2}]^{-1})$$
we get (\ref{rep}). Since $e^{-t |\xi|^2}[1-ze^{-t |\xi|^2}]^{-1}\in
L^1(\Real^d)\cap L^2(\Real^d)$, we have $g_z\in
L^\infty(\Real^d)\cap L^2(\Real^d)$.
\end{proof}

We note that while $\|G(z)\|_{L^2,L^2}=\|g_z\|_{L^1}\rightarrow
\infty$ when $z\rightarrow 1$, we are going to show - and this is a
key technical point for obtaining Theorem \ref{main} - that when
$d\geq 5$ the norm $\|g_z\|_{L^2}$ is in fact bounded for $|z|<1$.

We will denote, for $|z|<1$,
\begin{equation}\label{dm}M(z)=\|g_z\|_{L^2}.\end{equation}

We will need the following elementary estimates:
\begin{lemma}\label{esl}
Assuming $p>0$, $a<1$, Let
\begin{equation}\label{defj}J_p(a)=\int_1^{\infty}\frac{(\log(s))^{p-1}}{(s-a)^2}ds.\end{equation}
Then:

\begin{enumerate}
  \item For $p=2$ we have
  $$J_2(a)=O\Big(\log\Big(\frac{1}{1-a}\Big)\Big),\;\;as\;\;a\rightarrow1-$$
  \item For $p>2$ we have
  $$J_p(a)=O(1),\;\;as\;\;a\rightarrow1-$$
\end{enumerate}
\end{lemma}

\begin{proof}
We write
$$\int_1^{\infty}\frac{(\log(s))^{p-1}}{(s-a)^2}ds=\int_1^{2}\frac{(\log(s))^{p-1}}{(s-a)^2}ds
+\int_2^{\infty}\frac{(\log(s))^{p-1}}{(s-a)^2}ds.$$ The second
integral on the right-hand side is obviously finite and bounded
independently of $a\in (-\infty,1)$. We continue estimating the
first integral.

Assuming $p\geq 1$, and using the fact that $\log(s)\leq s-1$ for
$s\geq 1$, we have
\begin{eqnarray*}&&\int_1^{2}\frac{(\log(s))^{p-1}}{(s-a)^2}ds\leq
\int_1^{2}\frac{(s-1)^{p-1}}{(s-a)^2}ds=
\int_1^{2}\frac{(s-1)^{p-1}}{(s-a)^{3-p}(s-a)^{p-1}}ds \\&\leq&
\int_1^{2}\frac{(s-1)^{p-1}}{(s-a)^{3-p}(s-1)^{p-1}}ds=\int_1^{2}\frac{1}{(s-a)^{3-p}}ds\\
&=&\left\{
\begin{array}{cc}
                                              \log\Big(\frac{2-a}{1-a}\Big) & p=2\\
                                              \frac{1}{2-p}[(1-a)^{p-2}-(2-a)^{p-2}] & 1\leq p\neq 2 \\
                                              \end{array}\right.
                                              \;\;\;\;as
                                              \;\;a\rightarrow 1-.
\\&=&
\left\{
\begin{array}{cc}
                                              O\Big(\log\Big(\frac{1}{1-a}\Big)\Big) & p=2\\
                                              O(1) & p> 2\\
                                              \end{array}\right.
                                              \;\;\;\;as
                                              \;\;a\rightarrow 1-.
\end{eqnarray*}
\end{proof}

We now present our main estimate on $M(z)$.
\begin{lemma}\label{abh} Define $M(z)$ by (\ref{dm}).
\begin{enumerate}
\item If $d=4$ then for some $C>0$
$$M(re^{i\theta})\leq C\Big(\log\Big(\frac{1}{1-r\cos(\theta)}\Big)\Big)^{\frac{1}{2}}
\;\;\;\forall r\in [0,1),\;\theta\in [0,2\pi].$$

\item If $d\geq 5$ then $$\sup_{|z|<1}M(z)<\infty.$$
\end{enumerate}
\end{lemma}

\begin{proof}
We have
\begin{eqnarray}\label{bg1}M(z)
&=&|z|\|\mathfrak{F}^{-1}(e^{-t |\xi|^2}[1-ze^{-t
|\xi|^2}]^{-1})\|_{L^2}\\&=& |z|\|e^{-t |\xi|^2}[1-ze^{-t
|\xi|^2}]^{-1}\|_{L^2}=|z|\Big[\int_{\Real^d} \frac{e^{-2t
|\xi|^2}}{|1-ze^{-t
|\xi|^2}|^2}d\xi\Big]^{\frac{1}{2}}.\nonumber\end{eqnarray} From
(\ref{bg1}) one sees that $M(z)$ is uniformly bounded in the
complement of any neighborhood of the point $z=1$ in the unit disk,
so that the issue is to study the behavior of $M(z)$ when
$z\rightarrow 1$.
 It is easy to verify that, for any $|z|<1$ with $Re(z)>0$,
$\xi\in \Real^d$,
$$|1-ze^{-|\xi|^2}|\geq 1- Re(z)e^{-|\xi|^2},$$
hence
\begin{eqnarray}\label{sb}&&(M(z))^2=|z|^2\int_{\Real^d}
\frac{e^{-2t |\xi|^2}}{|1-ze^{-t|\xi|^2}|^2}d\xi\leq
|z|^2\int_{\Real^d}
\frac{e^{-2t |\xi|^2}}{(1-Re(z)e^{-t|\xi|^2})^2}d\xi\nonumber\\
&=&\omega_d |z|^2\int_0^{\infty}\frac{\rho^{d-1}e^{-2t
\rho^2}}{(1-Re(z)e^{-t\rho^2})^2}d\rho
=\omega_d\frac{|z|^2}{2}\frac{1}{t^{\frac{d}{2}}}\int_1^{\infty}\frac{1}{s}\frac{(\log(s))^{\frac{d}{2}-1}}{(s-Re(z))^2}
ds\nonumber\\&\leq&\omega_d\frac{|z|^2}{2}\frac{1}{t^{\frac{d}{2}}}\int_1^{\infty}\frac{(\log(s))^{\frac{d}{2}-1}}{(s-Re(z))^2}
ds=\omega_d\frac{|z|^2}{2}\frac{1}{t^{\frac{d}{2}}}
J_{\frac{d}{2}}(Re(z)),
\end{eqnarray}
where $\omega_d$ is the $d-1$-dimensional measure of the unit sphere
in $\Real^d$, and $J_{p}$ is defined by (\ref{defj}). The result
follows from (\ref{sb}) and from the estimates in Lemma \ref{esl}.
\end{proof}

\subsection{Some estimates on $D_t$}

We recall the Duhamel formula \begin{equation}\label{du}D_t=\int_0^t
e^{-s(H_0+V_-)}V_-e^{-(t-s)H_0}ds.\end{equation} The integral kernel
corresponding to the operator $D_t$ is denoted by $D_t(x,y)$.

The condition
\begin{eqnarray}\label{ddd}\int_{\Real^d}\Big(\int_{\Real^d}
|D_t(x,y)|dx\Big)^2dy <\infty.
\end{eqnarray}
on the kernel of $D_t$ will be essential to us, and will be used in
Lemma \ref{hs2} and \ref{prod}. The next lemma gives a sufficient
condition for (\ref{ddd}) to hold.

\begin{lemma}\label{edt210}
Assuming $V_-\in K(\Real^d)$, $t>0$ and
\begin{eqnarray}\label{u2}
\int_0^t\int_{\Real^d}\int_0^t
\int_{\Real^d}|V_-(w)|e^{-(s+s')H_0}(w,w') |V_-(w')| dw' ds' dw ds
<\infty,\end{eqnarray}
 we have (\ref{ddd}).
\end{lemma}

\begin{proof}
By the Duhamel formula (\ref{du}) we have
\begin{eqnarray*}&&D_t(x,y)=\int_0^t \int_{\Real^d}e^{-s(H_0+V_-)}(x,w)V_-(w)e^{-(t-s)H_0}(w,y)dw ds
\end{eqnarray*}
so
\begin{eqnarray*}&&\int_{\Real^d}|D_t(x,y)|dx\\&\leq&
\int_0^t\int_{\Real^d}
\Big(\int_{\Real^d}e^{-s(H_0+V_-)}(x,w)dx\Big)|V_-(w)|e^{-(t-s)H_0}(w,y)dw
ds \nonumber\\&\leq& \Big[\sup_{s\in
[0,t],\;w\in\Real^d}\int_{\Real^d}e^{-s(H_0+V_-)}(x,w)dx\Big]\int_0^t\int_{\Real^d}
|V_-(w)|e^{-sH_0}(w,y)dw ds,\\&=& \sup_{s\in [0,t]}
\|e^{-s(H_0+V_-)}\|_{L^\infty,L^\infty}\int_0^t\int_{\Real^d}
|V_-(w)|e^{-sH_0}(w,y)dw ds,\nonumber \nonumber \end{eqnarray*}
which implies
\begin{eqnarray*}
&&\Big(\int_{\Real^d}|D_t(x,y)|dx \Big)^2\leq \sup_{s\in [0,t]}
\|e^{-s(H_0+V_-)}\|_{L^\infty,L^\infty}^2\\&\times&\int_0^t\int_{\Real^d}\int_0^t
\int_{\Real^d}|V_-(w)|e^{-sH_0}(w,y) |V_-(w')|e^{-s'H_0}(w',y) dw'
ds' dw ds,
\end{eqnarray*}
hence
\begin{eqnarray}\label{u1}&&\int_{\Real^d}\Big(\int_{\Real^d}|D_t(x,y)|dx
\Big)^2dy\leq \sup_{s\in [0,t]}
\|e^{-s(H_0+V_-)}\|_{L^\infty,L^\infty}^2\\
&\times&\int_0^t\int_{\Real^d}\int_0^t
\int_{\Real^d}|V_-(w)|e^{-(s+s')H_0}(w,w') |V_-(w')| dw' ds' dw
ds.\nonumber
\end{eqnarray}
and the finiteness of the right-hand side of (\ref{u1}) follows from
 Lemma \ref{sch} and from the assumption (\ref{u2}).
\end{proof}

\begin{lemma}\label{hs2} If $V\in K(\Real^d)$  and (\ref{ddd}) holds for $t>0$
sufficiently small, then $D_t$ is Hilbert-Schmidt for $t>0$
sufficiently small.
\end{lemma}

\begin{proof}
By the identity
$$D_t=e^{-\frac{t}{2}(H_0+V_-)}D_{\frac{t}{2}}+D_{\frac{t}{2}}e^{-\frac{t}{2}H_0},$$
we have
\begin{eqnarray}\label{ff1}
\|D_t\|_{HS}&\leq &\|e^{-\frac{t}{2}(H_0+V_-)}D_{\frac{t}{2}}\|_{HS}+\|D_{\frac{t}{2}}e^{-\frac{t}{2}H_0}\|_{HS}\nonumber\\
&=&\|e^{-\frac{t}{2}(H_0+V_-)}D_{\frac{t}{2}}\|_{HS}+\|e^{-\frac{t}{2}H_0}D_{\frac{t}{2}}\|_{HS}.
\end{eqnarray}
Since
\begin{eqnarray}\label{q1}\|e^{-\frac{t}{2}(H_0+V_-)}D_{\frac{t}{2}}\|_{HS}^2=
\int_{\Real^d}\int_{\Real^d}\Big(\int_{\Real^d}
 e^{-\frac{t}{2}(H_0+V_-)}(x,u)D_{\frac{t}{2}}(u,y)du\Big)^2 dx
 dy,\nonumber\end{eqnarray}
$$\|e^{-\frac{t}{2}H_0}D_{\frac{t}{2}}\|_{HS}^2=\int_{\Real^d}\int_{\Real^d}\Big(\int_{\Real^d}
 e^{-\frac{t}{2}H_0}(x,u)D_{\frac{t}{2}}(u,y)du\Big)^2 dx dy,$$
 and since
 $$e^{-\frac{t}{2}H_0}(x,u)\leq e^{-\frac{t}{2}(H_0+V_-)}(x,u),$$
 we have
\begin{equation}\label{ff2}\|e^{-\frac{t}{2}H_0}D_{\frac{t}{2}}\|_{HS}\leq
\|e^{-\frac{t}{2}(H_0+V_-)}D_{\frac{t}{2}}\|_{HS}.\end{equation}
 From (\ref{q1})
we have,
\begin{eqnarray}\label{q2}&&\|e^{-\frac{t}{2}(H_0+V_-)}D_{\frac{t}{2}}\|_{HS}^2\nonumber\\&=&
\int_{\Real^d}\int_{\Real^d}\int_{\Real^d}\int_{\Real^d}
 e^{-\frac{t}{2}(H_0+V_-)}(x,u)D_{\frac{t}{2}}(u,y)e^{-\frac{t}{2}(H_0+V_-)}(x,u')D_{\frac{t}{2}}(u',y)du du' dx
 dy\nonumber\\
&=&\int_{\Real^d}\int_{\Real^d}\int_{\Real^d}
 e^{-t(H_0+V_-)}(u,u')D_{\frac{t}{2}}(u,y)D_{\frac{t}{2}}(u',y)du du'
 dy\nonumber\\
 &\leq&\Big[\sup_{x,y\in\Real^d}e^{-t(H_0+V_-)}(x,y)\Big]\int_{\Real^d}\Big(\int_{\Real^d}
 D_{\frac{t}{2}}(u,y)du\Big)^2 dy\nonumber\\
 &\leq& \|e^{-t(H_0+V_-)} \|_{L^1,L^\infty}\int_{\Real^d}\Big(\int_{\Real^d}
 D_{\frac{t}{2}}(u,y)du\Big)^2 dy< \infty,
\end{eqnarray}
where the finiteness of the two terms of the product on the
right-hand side follows from Lemmas \ref{sch} and \ref{edt210}.

The result follows from (\ref{ff1}), (\ref{ff2}) and (\ref{q2}).
\end{proof}

We now show that the condition (\ref{u2}) (which in turn, by Lemma
\ref{edt210}, implies the conditon (\ref{ddd}) which we need) is
implied by the explicit conditions on $V_-$ given in Theorem
\ref{main}.
\begin{lemma}\label{sc}
Assume that $V_-$ satisfies (\ref{cond0}) for some $c>0$.

\noindent (i) If $d=4$ and $V_-$ also satisfies (\ref{cond2}) then
(\ref{u2}) holds.

\noindent (ii) If $d\geq 5$ and $V_-$ also satisfies (\ref{cond1})
then (\ref{u2}) holds.
\end{lemma}
\begin{proof}
We assume that $V_-$ satisfies (\ref{cond0}) for some $c=c_0$, and
note that this implies that it satisfies (\ref{cond0}) for
{\it{all}} $c\geq c_0$.

 We have
\begin{eqnarray}\label{cal}&&\int_0^t\int_{\Real^d}\int_0^t
\int_{\Real^d}|V_-(w)|e^{-(s+s')H_0}(w,w') |V_-(w')| dw' ds' dw ds\\
&=&\int_0^{t}\int_0^u \int_{\Real^d} \int_{\Real^d}|V_-(w)|e^{-u
H_0}(w,w') |V_-(w')| dw dw' ds du\nonumber\\
&+&\int_t^{2t}\int_{u-t}^{t} \int_{\Real^d}
\int_{\Real^d}|V_-(w)|e^{-u H_0}(w,w') |V_-(w')| dw dw' ds
du=I_1+I_2,\nonumber\end{eqnarray} where
$$I_1=\int_0^{t} u \int_{\Real^d}
\int_{\Real^d}|V_-(w)|e^{-u H_0}(w,w') |V_-(w')| dw dw' du,$$
$$I_2=\int_0^{t}u \int_{\Real^d} \int_{\Real^d}|V_-(w)|e^{-(2t-u)
H_0}(w,w') |V_-(w')| dw dw' du.$$

For $I_2$ we have
$$I_2=
\int_{\Real^d} \int_{\Real^d}|V_-(w)||V_-(w')| \Big(\int_0^{t} u
(4\pi (2t-u))^{-\frac{d}{2}} e^{-\frac{|w-w'|^2}{4(2t-u)}}du\Big) dw
dw'$$
$$\leq t^2 (4\pi t)^{-\frac{d}{2}}
\int_{\Real^d} \int_{\Real^d} e^{-\frac{|w-w'|^2}{8t}}
|V_-(w)||V_-(w')| dw dw',
$$
which is finite for $t>0$ sufficiently small due to the assumption
(\ref{cond0}).

We are left with showing that $I_1$ is finite under the stated
conditions. We have
\begin{eqnarray}\label{fi}I_1=(4\pi)^{-\frac{d}{2}} \int_{\Real^d}
\int_{\Real^d}|V_-(w)||V_-(w')| \Big(\int_0^{t} u^{1-\frac{d}{2}}
e^{-\frac{|w-w'|^2}{4u}}du\Big)  dw dw'.\nonumber
\end{eqnarray}
and making the substitution $v=\frac{a}{u}$, we estimate
\begin{eqnarray}\label{gg}\int_0^t
u^{1-\frac{d}{2}}e^{-\frac{a}{u}}
du=a^{2-\frac{d}{2}}\int_{\frac{a}{t}}^\infty
v^{\frac{d}{2}-3}e^{-v} dv \leq
a^{2-\frac{d}{2}}e^{-\frac{a}{2t}}\int_{\frac{a}{t}}^\infty
v^{\frac{d}{2}-3}e^{-\frac{v}{2}} dv.
\end{eqnarray}
If $d\geq 5$ then
$$\int_{\frac{a}{t}}^\infty
v^{\frac{d}{2}-3}e^{-\frac{v}{2}} dv\leq \int_{0}^\infty
v^{\frac{d}{2}-3}e^{-\frac{v}{2}}<\infty,$$ so that, putting
$a=\frac{|w-w'|^2}{4}$ in (\ref{gg}), we have
$$\int_0^{t} u^{1-\frac{d}{2}}
e^{-\frac{|w-w'|^2}{4u}}du\leq C
\frac{e^{-\frac{|w-w'|^2}{8t}}}{|w-w'|^{d-4}},$$ where $C$ is
independent of $w,w'$, so that, by (\ref{fi}),
$$I_1\leq C\int_{\Real^d}
\int_{\Real^d}\frac{e^{-\frac{|w-w'|^2}{8t}}}{|w-w'|^{d-4}}|V_-(w)||V_-(w')|
dw dw'$$
$$\leq C
\iint_{|w-w'|\geq 1}e^{-\frac{|w-w'|^2}{8t}}|V_-(w)||V_-(w')| dw
dw'$$
$$+ C
\iint_{|w-w'|\leq 1}\frac{1}{|w-w'|^{d-4}}|V_-(w)||V_-(w')| dw dw'$$
and both of the last two integrals are finite, the first (for
sufficiently small $t>0$) by (\ref{cond0}) and the second by
(\ref{cond1}),
 so that, in the case $d\geq 5$, (\ref{cond1}) we have that
$I_1$ is finite.

To treat the case $d=4$ we note that, using L'H\^opital's rule, we
have
$$\lim_{\alpha \rightarrow
0+}\frac{1}{\log(\frac{1}{\alpha})} \int_{\alpha}^\infty
v^{-1}e^{-\frac{v}{2}}dv=1.$$ We can therefore choose $0<\alpha_0<1$
so that
$$0<\alpha\leq \alpha_0\;\;\Rightarrow \;\;\int_{\alpha}^\infty
v^{-1}e^{-\frac{v}{2}}dv\leq 2\log(\alpha^{-1})$$ and then we also
have
$$\alpha>\alpha_0\;\;\Rightarrow \;\;\int_{\alpha}^\infty
v^{-1}e^{-\frac{v}{2}}dv\leq \int_{\alpha_0}^\infty
v^{-1}e^{-\frac{v}{2}}dv\leq 2\log(\alpha_0^{-1}).$$ Therefore,
using (\ref{gg}),
$$0<a\leq \alpha_0 t\;\;\Rightarrow \;\;\int_0^t u^{-1}e^{-\frac{a}{u}}
du= e^{-\frac{a}{2t}}\int_{\frac{a}{t}}^\infty
v^{-1}e^{-\frac{v}{2}} dv \leq
2e^{-\frac{a}{2t}}\log\Big(\frac{t}{a}\Big),$$
$$a> \alpha_0 t\;\;\Rightarrow \;\;\int_0^t u^{-1}e^{-\frac{a}{u}}
du= e^{-\frac{a}{2t}}\int_{\frac{a}{t}}^\infty
v^{-1}e^{-\frac{v}{2}} dv \leq
2e^{-\frac{a}{2t}}\log(\alpha_0^{-1})$$ and setting
$a=\frac{|w-w'|^2}{4}$
$$0<|w-w'|\leq 2\sqrt{\alpha_0 t}\;\;\Rightarrow \;\;\int_0^t u^{-1}e^{-\frac{|w-w'|^2}{4u}} du\leq
2e^{-\frac{|w-w'|^2}{8t}}\log\Big(\frac{4t}{|w-w'|^2}\Big),$$
$$|w-w'|>2\sqrt{\alpha_0 t}\;\;\Rightarrow \;\;\int_0^t u^{-1}e^{-\frac{|w-w'|^2}{4u}} du\leq
2e^{-\frac{|w-w'|^2}{8t}}\log(\alpha_0^{-1}).$$ Hence, using
(\ref{fi}),
\begin{eqnarray*}
I_1&\leq&2\log(\alpha_0^{-1})(4\pi)^{-\frac{d}{2}}
\iint_{|w-w'|>2\sqrt{\alpha_0 t}} e^{-\frac{|w-w'|^2}{8t}}|V_-(w)||V_-(w')|  dw dw'\\
&+&4(4\pi)^{-\frac{d}{2}} \iint_{|w-w'|<2\sqrt{\alpha_0 t}}
\log\Big(\frac{2\sqrt{t}}{|w-w'|}\Big)|V_-(w)||V_-(w')| dw dw'.
\end{eqnarray*}
The first integral above is finite for $t$ sufficiently small, due
to (\ref{cond0}). For the second integral we have
$$ \iint_{|w-w'|<2\sqrt{\alpha_0
t}} \log\Big(\frac{2\sqrt{t}}{|w-w'|}\Big)|V_-(w)||V_-(w')| dw dw'$$
$$=\iint_{|w-w'|<2\sqrt{\alpha_0
t}} \log\Big(\frac{1}{|w-w'|}\Big)|V_-(w)||V_-(w')| dw dw'$$
$$+\log(2\sqrt{t}) \iint_{|w-w'|<2\sqrt{\alpha_0
t}}  |V_-(w)||V_-(w')| dw dw',$$ and finiteness of the above two
integrals for $t>0$ sufficiently small follows from (\ref{cond2})
and (\ref{cond0}), respectively. We have thus shown that $I_1$ is
finite when $d=4$.
\end{proof}

\subsection{Hilbert-Schmidt norm bound for the composition of $G(z)$
and $D_t$}

\begin{lemma}\label{prod}
Assume $V_-\in K(\Real^d)$ and that (\ref{ddd}) holds. Then
$G(z)D_t$ is Hilbert-Schmidt, and
$$\|G(z)D_t\|_{HS}\leq  M(z) \Big[\int_{\Real^d}\Big(\int_{\Real^d}
|D_t(u,y)|du\Big)^2dy\Big]^{\frac{1}{2}},$$
\end{lemma}

\begin{proof}
We have
$$[G(z)D_t](x,y)=\int_{\Real^d}g_z(x-u)D_t(u,y)du,$$
hence
$$([G(z)D_t](x,y))^2=\int_{\Real^d}\int_{\Real^d}g_z(x-u)D_t(u,y)g_z(x-v)D_t(v,y)du dv$$
and thus
\begin{eqnarray}\label{g1}&&\|G(z)D_t\|_{HS}^2=\int_{\Real^d}\int_{\Real^d}([G(z)D_t](x,y))^2dx
dy\\
&=&\int_{\Real^d}\int_{\Real^d}\Big(\int_{\Real^d}g_z(x-u)g_z(x-v)dx\Big)
\Big(\int_{\Real^d}D_t(u,y)D_t(v,y)dy\Big)du dv\nonumber\\
&\leq& \Big[\sup_{u,v\in\Real^d}
\Big(\int_{\Real^d}|g_z(x-u)g_z(x-v)|dx\Big)\Big] \int_{\Real^d}
\int_{\Real^d} \int_{\Real^d}|D_t(u,y)D_t(v,y)|dydu dv.\nonumber\\
&=& \Big[\sup_{u,v\in\Real^d}
\Big(\int_{\Real^d}|g_z(x-u)g_z(x-v)|dx\Big)\Big]\int_{\Real^d}\Big(\int_{\Real^d}
|D_t(u,y)|du\Big)^2dy.\nonumber
\end{eqnarray}
We also have
\begin{eqnarray}\label{g3}&&\sup_{u,v\in\Real^d}\int_{\Real^d} |g_z(x-u)g_z(x-v)|dx\\&\leq&
\sup_{u,v\in\Real^d}\Big(\int_{\Real^d}|g_z(x-u)|^2
dx\Big)^{\frac{1}{2}} \Big(\int_{\Real^d}|g_z(x-v)|^2
dx\Big)^{\frac{1}{2}}=\|g_z\|_{L^2}^2 = (M(z))^2\nonumber.
\end{eqnarray}
The result follows from (\ref{g1}) and (\ref{g3}).
\end{proof}

\subsection{Bounding the Jensen integral}

\begin{lemma}\label{bob}Assume $V_-\in K(\Real^d)$ and (\ref{ddd}) holds. Then we have, for all
$|z|<1$,
\begin{eqnarray*}\|F(z)\|_{HS}&\leq& |z|[
C_1+ C_2 M(z)].
\end{eqnarray*}
where
$$C_1=\|D_t\|_{HS},\;\;\;C_2=\Big[\int_{\Real^d}\Big(\int_{\Real^d}
|D_t(u,y)|du\Big)^2dy\Big]^{\frac{1}{2}}.$$
\end{lemma}

\begin{proof}
Returning to (\ref{fo}) and using Lemmas \ref{hs2} and \ref{prod} we
get the result.
\end{proof}

We are now ready for
\begin{proof1}
By the above lemma and by (\ref{itt}) we have
\begin{equation}\label{cb}\log(|h(z)|)\leq \frac{1}{2} \|F(z)\|_{HS}^2\leq  \frac{1}{2}|z|^2[
C_1+ C_2 M(z)]^2.\end{equation}

In the case $d=4$, (\ref{cb}) and Lemma \ref{abh}(a) give us
$$\log(|h(re^{i\theta})|)\leq \frac{1}{2} \Big[ C_1+ C_2
C\Big(\log\Big(\frac{1}{1-r\cos(\theta)}\Big)\Big)^{\frac{1}{2}}\Big]^2
\leq C'\log\Big(\frac{1}{1-r\cos(\theta)}\Big),$$ so that
$$\limsup_{r\rightarrow 1-}\int_0^{2\pi}\log(|h(re^{i\theta})|)d\theta\leq
C'
\int_0^{2\pi}\log\Big(\frac{1}{1-\cos(\theta)}\Big)d\theta<\infty,$$
which again, by Theorem \ref{identity}, gives us (\ref{m}).

In the case $d\geq 5$, Lemma \ref{abh}(c) tells us that
$$\log(|h(re^{i\theta})|)\leq C\;\;\;\forall r\in [0,1),\;\;\theta\in [0,2\pi]$$
so that
$$\limsup_{r\rightarrow
1-}\int_0^{2\pi}\log(|h(re^{i\theta})|)d\theta<\infty$$ and again
Theorem \ref{identity} gives us (\ref{m}).
\end{proof1}

We now prove the corollaries.
\begin{proof2}
We use Young's inequality:
$$\|f\ast g\|_{L^r}\leq C \|f\|_{L^p} \|g\|_{L^q}$$
valid for $p,q,r\geq 1$ with
$\frac{1}{r}+1=\frac{1}{p}+\frac{1}{q}$.

Taking $f(x)=|V_-(x)|$, $g(x)=e^{-c|x|^2}$, and
\begin{equation}\label{ay}p\in [1,2],\;\;
q=\frac{p}{2(p-1)},\;\;r=\frac{p}{p-1}\end{equation} (note that
since $p\leq 2$ we have $q\geq 1$), we have that if $V_-\in
L^p(\Real^d)$ then $f\ast g\in L^{\frac{p}{p-1}}(\Real^d)$. Thus,
using H\"older's inequality we have
$$\iint_{\Real^d}e^{-c|w-w'|^2}|V_-(w)||V_-(w')| dw'
dw$$$$ = \int_{\Real^d} |V_-(w)|(f\ast g)(w)dw \leq
\|V_-\|_{L^p}\|f\ast g\|_{L^\frac{p}{p-1}}<\infty.
$$
Thus (\ref{cond0}) holds whenever $V_-\in L^p(\Real^d)$, $p\in
[1,2]$.

To verify (\ref{cond2}) (for the case $d=4$) we take $f=|V_-|$,
$g(x)=\log(\frac{1}{|x|})\chi_{B_1}$, where $\chi_{B_1}$ is the
characteristic function of the unit ball. and $p,q,r$ according to
(\ref{ay}). We assume $V_-\in L^p$, and note that $g\in
L^q(\Real^4)$. Hence by Young's inequality we have $f\ast g\in
L^{\frac{p}{p-1}}(\Real^4).$ Therefore, using H\"older's inequality,
we have
$$
\iint_{|w-w'|<1}\log\Big(\frac{1}{|w-w'|}\Big)|V_-(w)||V_-(w')| dw'
dw$$
$$= \int_{\Real^4}|V_-(w)|(f\ast g)(w)dw \leq
\|V_-\|_{L^p}\|f\ast g\|_{L^\frac{p}{p-1}}<\infty,$$ so that
(\ref{cond2}) is satisfied for any $V_-\in L^p(\Real^4)$, $p\in
(1,2]$.

To verify (\ref{cond1}) (for the case $d\geq 5$) we take $f=|V_-|$,
$g(x)=\frac{1}{|x|^{d-4}}\chi_{B_1}$, and $p,q,r$ defined by
(\ref{ay}). Note that in order to have $g\in L^q(\Real^d)$ we need
the condition: $q(d-4)<d$, that is $p>\frac{2d}{d+4},$ and in that
case we get, assuming $V_-\in L^p$, $\|f\ast
g\|_{L^\frac{p}{p-1}}<\infty$, and thus
$$\iint_{|w-w'|<1}\frac{|V_-(w)||V_-(w')|}{|w-w'|^{d-4}}
dw dw' $$$$=  \int_{\Real^d}|V_-(w)|(f\ast g)(w)dw \leq
\|V_-\|_{L^p}\|f\ast g\|_{L^\frac{p}{p-1}}<\infty.$$ Thus we have
shown that (\ref{cond1}) holds when $p>\frac{2d}{d+4}.$ To show that
it also holds when $p=\frac{2d}{d+4}$, we need to use the
Hardy-Littlewood-Sobolev inequality (see e.g. \cite{liebloss},
Theorem 4.3), which says that
$$\Big|\int_{\Real^d}\int_{\Real^d} \frac{f(w)g(w')}{|w-w'|^{\lambda}} dw dw' \Big|\leq C \|f\|_{L^p} \|g\|_{L^r},$$
where $p,r>1$, $0<\lambda<d$, and
$\frac{1}{p}+\frac{1}{r}=2-\frac{\lambda}{d}$. We take
$p=r=\frac{2d}{d+4}$ (since $d\geq 5$ we have $p,r>1$),
$\lambda=d-4$, $f=g=|V_-|$, to obtain
$$\int_{\Real^d}\int_{\Real^d}\frac{|V_-(w)||V_-(w')|}{|w-w'|^{d-4}} dw dw'
<\infty,$$ which is even stronger than (\ref{cond1}).
\end{proof2}

\begin{proof3}
Since Corollary \ref{cor1} contains the result for the case $d=4$,
we can assume that $d\geq 5$. Since $V_-\in L^1(\Real^d)$, we can
write
$$\iint_{|w-w'|<1}\frac{|V_-(w)||V_-(w')|}{|w-w'|^{d-4}}
dw dw' \leq \|V\|_{L^1}\sup_{w\in \Real^d}
\int_{|w-w'|<1}\frac{|V_-(w')|}{|w-w'|^{d-4}} dw'. $$ Thus we only
need to show that the supremum on the right-hand side is finite.
Since $V_-\in K(\Real^d)$, we have from (\ref{kato}) that there
exists $\alpha<1$ such that
$$\sup_{w\in\Real^d}
\int_{|w-w'|\leq \alpha}\frac{|V(w')|}{|w-w'|^{d-2}}dw'=\eta<
\infty.$$ Thus
$$\sup_{w\in \Real^d}
\int_{|w-w'|<1}\frac{|V_-(w')|}{|w-w'|^{d-4}} dw'$$
$$\leq \sup_{w\in\Real^d}
\int_{|w-w'|\leq \alpha}\frac{|V_-(w')|}{|w-w'|^{d-4}}dw'
+\sup_{w\in\Real^d} \int_{\alpha <|w-w'|\leq
1}\frac{|V_-(w')|}{|w-w'|^{d-4}}dw'$$
$$\leq \sup_{w\in\Real^d}
\int_{|w-w'|\leq \alpha}\frac{|V_-(w')|}{|w-w'|^{d-2}}dw'
+\frac{1}{\alpha^{d-4}}\sup_{w\in\Real^d} \int_{\alpha <|w-w'|\leq
1}|V_-(w')|dw'$$
$$\leq \eta + \frac{1}{\alpha^{d-4}} \|V_-\|_{L^1},$$
and we have proved the required finiteness.
\end{proof3}

{\bf{Acknowledgment:}} We thank M. Hansmann for many
 helpful discussions.

\end{document}